\documentclass[12pt]{amsart}

\usepackage{amsmath}
\usepackage{amssymb}  
\usepackage{latexsym} 
\usepackage{comment}
\usepackage{url}
\usepackage{fullpage,url,amssymb,amsmath,amsthm,amsfonts,mathrsfs}
\usepackage[usenames,dvipsnames]{color}
\usepackage[pagebackref = true, colorlinks = true, linkcolor = blue, citecolor = Green]{hyperref}
\usepackage[alphabetic,lite]{amsrefs}
\usepackage{enumitem}
\usepackage{amscd}   
\usepackage[all, cmtip]{xy} 
\usepackage{xfrac}
\usepackage[T1]{fontenc}

\DeclareFontEncoding{OT2}{}{} 
\newcommand{\textcyr}[1]{%
 {\fontencoding{OT2}\fontfamily{wncyr}\fontseries{m}\fontshape{n}\selectfont #1}}

\usepackage[all]{xy}
\usepackage{fullpage}
\newcommand{\Sha}{{\mbox{\textcyr{Sh}}}}

\usepackage{color} 

\newcommand{\defi}[1]{\textsf{#1}} 

\def\act#1#2%
  {\mathop{}%
   \mathopen{\vphantom{#2}}^{#1}%
   \kern-3\scriptspace%
   #2}

\newcommand{\Z}{{\mathbb Z}}

\newcommand{\F}{{\mathbb F}}
\newcommand{\A}{{\mathbb A}}
\newcommand{\PP}{{\mathbb P}}

\newcommand{\vbar}{{\overline{v}}}
\newcommand{\Abar}{{\overline{A}}}

\newcommand{\Dbar}{{\overline{D}}}

\newcommand{\Fbar}{{\overline{\F}}}

\newcommand{\calO}{{\mathcal O}}

\newcommand{\To}{\longrightarrow}

\DeclareMathOperator{\Cov}{Cov}

\DeclareMathOperator{\Map}{Map}

\DeclareMathOperator{\Hom}{Hom}

\DeclareMathOperator{\Gal}{Gal}

\DeclareMathOperator{\Br}{Br}

\DeclareMathOperator{\Pic}{Pic}

\DeclareMathOperator{\Jac}{Jac}
\DeclareMathOperator{\HH}{H}

\DeclareMathOperator{\Spec}{Spec}

\DeclareMathOperator{\Mor}{Mor}

\newcommand{\etale}{\operatorname{\textup{\'et}}}
\newcommand{\fab}{\operatorname{ab}}

\newcommand{\isog}{\operatorname{isog}}
\newcommand{\etaleisog}{\operatorname{\textup{\'et}-isog}}

\newtheorem{Theorem}{Theorem}[section]
\newtheorem{Lemma}[Theorem]{Lemma}
\newtheorem{Proposition}[Theorem]{Proposition}
\newtheorem{Corollary}[Theorem]{Corollary}

\newtheorem{Remark}[Theorem]{Remark}

\newtheorem{Question}[Theorem]{Question}

\theoremstyle{definition}

\numberwithin{equation}{section}

\begin{document}
\title{The Brauer-Manin obstruction for constant curves over global function fields}
\author{Brendan Creutz}
\address{School of Mathematics and Statistics, University of Canterbury, Private Bag 4800, Christchurch 8140, New Zealand}
\email{brendan.creutz@canterbury.ac.nz}
\urladdr{http://www.math.canterbury.ac.nz/\~{}b.creutz}

\author{Jos\'e Felipe Voloch}
\address{School of Mathematics and Statistics, University of Canterbury, Private Bag 4800, Christchurch 8140, New Zealand}
\email{felipe.voloch@canterbury.ac.nz}
\urladdr{http://www.math.canterbury.ac.nz/\~{}f.voloch}

\begin{abstract}
Let $\F$ be a finite field and $C,D$ smooth, geometrically irreducible, proper curves over $\F$ and set $K = \F(D)$. We consider Brauer-Manin and abelian descent obstructions to the existence of rational points and to weak approximation for the curve $C \otimes_\F K$. In particular, we show that Brauer-Manin is the only obstruction to weak approximation and the Hasse principle in the case that the genus of $D$ is less than that of $C$. We also show that we can identify the points corresponding to non-constant maps $D \to C$ using Frobenius descents.
\end{abstract}

\maketitle

\section{Introduction}

Let $C$ be a smooth, geometrically irreducible, proper curve of genus $>1$ over a global field $K$.
The question of whether the Brauer-Manin obstruction is the only obstruction
to the Hasse principle or weak approximation for $C$ was raised around 1999
by Scharaschkin and Skorobogatov. Not much progress has been made in the number field case, but a substantial 
amount of numerical evidence has been obtained, notably \cite{BruinStoll}.
By contrast, the question in the function field case has been settled
affirmatively for ``most'' curves \cite{PoonenVoloch}, namely those curves
whose Jacobian does not have an isotrivial factor and satisfy a certain
condition on the $p$-power torsion points, where $p$ is the characteristic of $K$.
The latter condition has largely been removed due to recent work of R\"ossler
\cite{Rossler1}. 
However, these results do not address the case of isotrivial or
even constant curves and the purpose of this paper is to consider this case. 

To put our results in context we begin with a summary of the general situation, assuming that $C$ is embedded in its Jacobian $J$ and that $J$ has finite Tate-Shafarevich group. If $S_C$ is the set of primes of good reduction for (some model of) $C$, then there is a commutative diagram
\[
	\xymatrix{
		C(K) \ar@{.>}[r]\ar@{.>}[d] & J(K) \ar[d] \\
		\prod_{v \in S_C} C(\F_v) \ar[r] & \prod_{v \in S_C} J(\F_v)\,,
	}
\]
where $\F_v$ denotes the residue field at the prime $v$ of $K$. Scharaschkin \cite{Scharaschkin} considered the set $C^\textup{MW-Sieve}$, which is the intersection of the topological closures of the images of $J(K)$ and $\prod_{v \in S_C}C(\F_v)$ inside $\prod_{v \in S_C}J(\F_v)$. He showed that if $C^\textup{MW-Sieve}$ is empty, then the set $C(\A_K)^{\Br}$ of adelic points orthogonal to the Brauer group of $C$ is too, and used this to give examples of curves of genus at least $2$ which are counterexamples to the Hasse principle explained by the Brauer-Manin obstruction. Poonen conjectured, in the number field case, that every counterexample to the Hasse principle could be explained this way \cite{Poonen}. Around the same time Stoll conjectured, in the number field case, that the set $C(\A_K)_\bullet^{\Br}$ of adelic points, modified at archimedean primes, orthogonal to the Brauer group of $C$ is equal to the (topological closure of the) image of $C(K)$ \cite{Stoll}. Recall that $C(K)$ can be infinite when $C$ is an isotrivial curve over a global function field,  even if the genus of $C$ is greater than $1$. For an example showing that it is in general necessary to take the topological closure see \cite[Remark 1.3]{PoonenVoloch}. One may also ask if $C^\textup{MW-Sieve}$ is equal to the topological closure of the image of $C(K)$ in $\prod_{v \in S_C}J(\F_v)$. There are logical dependencies between these statements as follows.
\begin{align}\label{conjectures}
\xymatrix{
	\fbox{$\,\overline{C(K)} = C(\A_K)_\bullet^{\Br} \text{ in } C(\A_K) \,$} \ar@{=>}[r] & \fbox{$\,C(K)= \emptyset \Rightarrow C(\A_K)^{\Br} = \emptyset \,$} \\
	\fbox{$\,\overline{C(K)} = C^\textup{MW-Sieve} \text{ in } \prod_{v \in S_C}C(\F_v)\,$} \ar@{=>}[r] & \fbox{$\,C(K)= \emptyset \Rightarrow C^\textup{MW-Sieve} = \emptyset \,$}\ar@{=>}[u] 
}
\end{align}
All four statements are known to hold when $J(K)$ is finite by a result of Scharaschkin \cite{Scharaschkin}. The statement in the upper-left is the conclusion of the theorem of Poonen-Voloch mentioned above; the theorem says nothing about the statement in the lower-right.

\subsection{Results for constant curves over global function fields}

Let $\F$ be a finite field and let $C,D$ be smooth, geometrically irreducible, proper curves over $\F$. Set $K = \F(D)$. We consider Brauer-Manin and finite descent obstructions to the existence of $K$-rational points on the curve $C \otimes_\F K$, which we also denote by $C$. We remark that $C$ may be embedded in its Jacobian $J$ since it has a $0$-cycle of degree $1$ defined over $\F$ by the Hasse-Weil bounds and that the Tate-Shafarevich group of $J$ is finite by results of Tate and Milne \cite[Theorem 3]{Milne}.

Let $C(\A_{K,\F}) := \prod_v C(\F_v) \subset C(\A_K)$ be the set of reduced adelic points on $C$ (see Section~\ref{sec:redadeles} for details). Since $C$ can be defined over $\F$ we have a model with everywhere good reduction and so $C^\textup{MW-Sieve} \subset C(\A_{K,\F}) \subset C(\A_K)$.

\begin{Theorem}\label{Thm1}
	$C(\A_K)^{\Br} = C(K) \cup C^\textup{MW-Sieve}$.
\end{Theorem}

In Corollary~\ref{cor:conjectures} we will deduce from Theorem~\ref{Thm1} that the statements on the left in~\eqref{conjectures} are equivalent in the case of constant curves over global function fields (and similarly for those on the right, though this was already known \cite[Proposition 2.2]{CVV}). 

Our proof of Theorem~\ref{Thm1} utilizes the connection between the Brauer-Manin and finite abelian descent obstructions developed in \cite{Stoll} building on work of Colliot-Th\'el\`ene and Sansuc, Harari and Skorobogatov. In particular, $C(\A_K)^{\Br}$ is the set of adelic points surviving all torsors $C' \to C$ arising as pullbacks of isogenies $J \to J$ (See Proposition~\ref{prop:BMandDescent}). So Theorem~\ref{Thm1} is a consequence of the following two theorems proven in Sections~\ref{sec:frob} and~\ref{sec:etale}, respectively.

\begin{Theorem}\label{Thm:FrobDescent}
	Let $C(\A_K)^{F^\infty}$ denote the set of adelic points surviving the $n$-th iterate of the $\F$-Frobenius isogeny $F:J \to J$ for all $n \ge 1$. Then $$C(\A_K)^{F^\infty} = C(K) \cup C(\A_{K,\F}).$$
\end{Theorem}

\begin{Theorem}\label{Thm:EtDescent}
	Let $C(\A_K)^{\etaleisog}$ denote the set of adelic points surviving all torsors arising as pullbacks of \'etale isogenies $J' \to J$. Then
	$$C(\A_{K})^{\etaleisog} \cap C(\A_{K,\F}) = C^\textup{MW-Sieve}.$$
\end{Theorem}

This result remains true if $C(\A_K)^{\etaleisog}$ is replaced by the a priori smaller set of adelic points surviving all torsors under abelian group schemes over $K$ (see Theorem \ref{thm:BrAKF}). Propositions~\ref{prop:example1} and~\ref{prop:example2} show that in general non-\'etale torsors are required to cut out the set of rational points (even if one includes non-abelian group schemes). 

In Corollary~\ref{cor:abdescC} we give a characterization of the elements of $C^\textup{MW-Sieve}$ in terms of maps between the sets $D(\Fbar) \to C(\Fbar)$ which induce homomorphisms of their Jacobians. We are thus lead to ask if all such maps arise from a global point. 

\begin{Question}
	Suppose $\psi : D(\Fbar) \to C(\Fbar)$ is a Galois-equivariant map of sets which, when extended linearly to divisors, sends principal divisors to principal divisors. Is there a morphism of curves $\phi:D \to C$ such that $\psi$ is given by composing $\phi$ with a limit of Frobenius maps?
\end{Question}

An affirmative answer to this question implies that all four statements in~\eqref{conjectures} hold in the case of constant curves over global function fields. In Section~\ref{sec:maps} we show that the answer to this question is affirmative when the genus of $D$ is less than the genus of $C$.

\begin{Theorem}\label{thm:Brgenus}
	If the genus of $D$ is less than the genus of $C$, then $C(\A_K)^{\Br} = C(K) = C(\F)$.
\end{Theorem}

	\section{Notation and preliminaries}
	
		As in the introduction $C$ will denote a constant curve over the global function field $K = \F(D)$ and $J$ is the Jacobian of $C$. The places of $K$ are in bijection with the set $D^1$ of closed points of $D$. Given a closed point $v \in D^1$ we use $K_v$, $\mathcal{O}_v$ and $\F_v$ to denote the corresponding completion, ring of integers and residue field, respectively. Let us fix once and for all an algebraic closure $\Fbar$ of $\F$ and for each place $v$, an embedding $\F \subset \F_v \subset \Fbar$. The embedding determines a geometric point $\overline{v} \in D(\Fbar)$ in the support of $v$. Let $G_\F = \Gal(\Fbar/\F)$ be the absolute Galois group of $\F$.
	
	Throughout the paper $X$ denotes a proper geometrically integral variety over $\F$ and $X_K := X \otimes_\F K$. Unless otherwise specified, all cohomology is flat cohomology.
	
	\subsection{Adelic points}
	The \defi{adele ring of $K$} is the $K$-algebra $\A_K = \prod_{v \in D^1} (K_v:\mathcal{O}_v)$, where the restricted product runs over the closed points of $D$. For any place $v$ of $K$, the inclusions $\F \subset \F_v \subset \mathcal{O}_v \subset K_v$ endow $\mathcal{O}_v, K_v$ and $\A_K$ with the structure of $\F$-algebra. Consequently the sets
\begin{align*}
	X(K_v) &:= \Mor_{\Spec(\F)}(\Spec(K_v),X)\,,\\
	X(\A_K) &:= \Mor_{\Spec(\F)}(\Spec(\A_K),X)
\end{align*}
are well defined. The universal property of fibered products gives canonical bijections of these sets with $X_{K}(K_v)$ and $X_K(\A_K)$, respectively. Since $X$ is proper we may identify $X(\A_K) = X_K(\A_K) = \prod_{v \in D^1}X(K_v) = \prod_{v\in D^1}X(\mathcal{O}_v)$.

\subsection{Reduced adelic points}\label{sec:redadeles}
The \defi{reduced adele ring of $K$} is the $\F$-algebra $\A_{K,\F} = \prod_{v\in D^1} \F_v$. This is an $\F$-subalgebra of $\A_K$. The set $X(\A_{K,\F}) = \Mor_{\Spec(\F)}(\Spec(\A_{K,\F}),X)$ of reduced adelic points on $X$ is a closed subset of $X(\A_{K})$ which can be identified with $\prod_{v\in D^1} X(\F_v)$, where the latter is endowed with the product of the discrete topologies. This agrees with the subspace topology determined by the inclusion $X(\A_{K,\F}) \subset X(\A_{K})$. The quotient of $\mathcal{O}_v$ by its maximal ideal induces the reduction map $r_v : X(K_v) = X(\mathcal{O}_v) \to X(\F_v)$. These give rise to a continuous projection $r : X(\A_K) \to X(\A_{K,\F})$ sending $(x_v)$ to $(r_v(x_v))$. 

\begin{Lemma}\label{lem:maps}
	A reduced adelic point $(x_v) \in X(\A_{K,\F})$ determines a unique map of $G_\F$-sets $\psi : D(\Fbar) \to X(\Fbar)$ with the property that $\psi(\overline{v}) = x_v$. This induces a bijection $X(\A_{K,\F}) \leftrightarrow \Map_{G_\F}(D(\Fbar),X(\Fbar))$. 
\end{Lemma}

\begin{proof}
	Since $D(\Fbar)$ is the union, as $v$ ranges over the closed points of $D$, of the $G_\F$-orbits of the points $\vbar$, it is clear that there is a unique $G_\F$-equivariant map with the stated property. Conversely, given a map of $G_\F$-sets $\psi : D(\Fbar) \to X(\Fbar)$, we define an adelic point $(x_v) \in X(\A_{K,\F})$ by $x_v = \psi(\vbar) \in X(\Fbar)$. Galois equivariance of the map shows that $x_v \in X(\F_v)$.
\end{proof}

	\subsection{Rational points}
		
	The universal property of fibered products and the valuative criterion for properness give identifications $X_K(K) = X(K) = \Mor_{\F}(\Spec(K),X) = \Mor_{\F}(D,X)$. Together with previous lemma we have a commutative diagram
\[
	\xymatrix{
		X(K) \ar[r]^r \ar@{=}[d] & X(\A_{K,\F}) \ar@{=}[d] \\
		\Mor_\F(D,X) \ar[r] & \Map_{G_\F}(D(\Fbar),X(\Fbar))
	}
\]
where the bottom map is the obvious one taking a morphism of varieties to the map it induces on geometric points. Since a morphism of varieties over a field with geometrically reduce source is determined by what it does to geometric points \cite[Exercise 5.16]{AGI}, the horizontal maps are injective.

\subsection{Brauer-Manin and abelian descent obstructions}

	Consider the category $\Cov(X_K)$ of $X_K$-torsors under finite group schemes over $K$ (see \cite[Section 4]{Stoll}).

	We say that an adelic point $P \in X(\A_{K}) = X_K(\A_K)$ survives $(X',G) \in \Cov(X_K)$ if the element of $\prod_v \HH^1(K_v,G)$ given by evaluating $(X',G)$ at $P$ lies in the image of the diagonal map $\HH^1(K,G) \to \prod_v\HH^1(K_v,G)$. Equivalently $P$ survives $(X',G)$ if and only if $P$ lifts to an adelic point on some twist of $(X',G)$. The set of adelic points surviving a set of torsors is a subset of $X(\A_K)$ containing $K$. Let $X(\A_K)^{\fab}$ denote the set of adelic points surviving all $(X',G) \in \Cov(X_K)$ for which $G$ is a finite abelian group scheme over $K$.
	
	When $X$ is a subvariety of an abelian variety $A/\F$ we shall consider the subset of torsors in $\Cov(X_K)$ which arise as pullbacks of (\'etale) isogenies $\phi:A' \to A$ defined over $\F$. We note that these are geometrically connected torsors under finite abelian group schemes over $\F$. They depend on the embedding $X \to A$, but only up to twist by elements of $\HH^1(\F,\ker(\phi))$. As such the sets $X(\A_K)^{\isog}$ and $X(\A_K)^{\etaleisog}$ of adelic points surviving all such torsors do not depend on the embedding.
		
	\begin{Proposition}\label{prop:BMandDescent}
		Suppose $X/\F$ is either a curve or an abelian variety. Then $X(\A_K)^{\Br} = X(\A_K)^{\fab} = X(\A_K)^{\isog}$. If $X$ is an abelian variety, then $X(\A_K)^{\Br} = \overline{X(K)}$.
	\end{Proposition}
	
	\begin{proof}
		Stoll proved the number field analogue of the first statement \cite[Section 7]{Stoll}. For the extension to global function fields see \cite[Section 2]{CVV}. For the second statement see \cite[Remark 4.5]{PoonenVoloch}. Note that this uses \cite{GAT}, which in turn uses crucially \cite[Th III.8.2]{MilneADT} a complete proof of which can be found in \cite{DH}.
	\end{proof}

\section{Frobenius descent obstruction}\label{sec:frob}

	Let $J$ be the Jacobian of $C$ and fix an embedding $C \to J$. We will consider the relative Frobenius morphism $F:J^{(-1)} \to J$ constructed as follows. Let $J^{(-1)} \to \Spec\F$ be the pullback of $J \to \Spec \F$ by the map on $\Spec \F$ induced by raising elements of $\F$ to $p^{(n-1)}$, where $|\F| = p^n$. Then $J$ is the pullback of $J^{(-1)} \to \Spec \F$ by the $p$th power map on $\Spec \F$. The universal property of the fibered product gives a morphism $\F$-schemes $F: J^{(-1)} \to J$. Zariski locally, defining equations for $J$ are obtained from those defining $J^{(-1)}$ by taking $p$-th powers and $F$ is given by raising coordinates to their $p$-th powers. 

	The pullback of $F$ along the embedding $C \to J$ yields a torsor $(C',\ker(F)) \in \Cov(C_K)$ under the finite abelian $K$-group scheme $\ker(F) \subset J^{(-1)}$ (which is strictly speaking the base change of $\ker(F)$ to $K$). The torsor $(C',\ker(F))$ does not depend on the choice of embedding because $F: J^{(-1)}(\F) \to J(\F)$ is surjective. We note that $C'$ is not reduced; the induced morphism $C'_\textup{red} \to C$ on the reduced subscheme of $C'$ is the $\F_p$-Frobenius morphism $C^{(-1)} \to C$ which has degree $p$, while $C' \to C$ has degree $p^{g}$.

	For a separable extension $L/K$ the Kummer sequence associated to $F$ is an exact sequence of flat cohomology groups
	\[
		0 \to J(L)/F(J^{(-1)}(L)) \stackrel{\delta_{F,L}} \To \HH^1(L,\ker F) \To \HH^1(L,J)[F] \to 0\,.
	\]
	The connecting morphism $\delta_{F,L}$ has the following explicit description.
	
\begin{Lemma}\label{lem:deltaFL}
There is a canonical injective group homomorphism $\Phi_L: H^1(L,\ker F) \to \Hom(\Omega^1_{J/L},\Omega_{L/\F})$ which is functorial on separable extensions $L/K$ and such that the composition $\Phi_L \circ \delta_{F,L}$ sends $x \in J(L)$ to the map $\left( \omega \mapsto x^*\omega\right)$.
\end{Lemma}

The existence of the map is a slight restatement of \cite[Proposition 1.1]{Rossler2} (which is, in turn, a variant of \cite[\S2]{AM}) and the explicit expression follows from the proof there. See also \cite[pg 123-124]{BV} for a proof in the ordinary case.

\begin{Lemma}
\label{lem:Frob}
	If $(x_v) \in C(\A_K)$ survives $(C',\ker(F))$ and $\delta_{F,K_v}(x_v) \ne 0$ for some $v$, then $(x_v) \in C(K)$ unless $p = 2$ and $C$ is hyperelliptic, in which case a similar result holds with $F$ replaced by $F^2$.
\end{Lemma}

\begin{proof}
Let $\omega_1,\dots,\omega_g$ be a basis of holomorphic differentials of $C$ over $\F$ (so of $J$ as well). Choose a separating variable $t \in \F(C)$ and write $\omega_i = f_i dt$ with $f_i \in \F(C)$.  This determines, for each separable $L/K$, an isomorphism $\Hom(\Omega^1_{J/L},\Omega_{L/\F}) \simeq \Omega^{\oplus g}_{L/\F}$. Let $\mu_L : J(L) \to \Omega^{\oplus g}_{L/\F}$ be the map obtained by composing with the map given by Lemma~\ref{lem:deltaFL}. If $\delta_{F,K_v}(x_v) \ne 0$, then the image of $\mu_{K_v}(x_v)$ in $\PP(\Omega_{K_v/\F}^{\oplus g}) = \PP^{g-1}(K_v)$ is the point $(f_1(x_v) : \dots : f_g(x_v))$, which is the image of $x_v$ under the canonical map $C \to \PP^{g-1}$. 

If $x=(x_v) \in C(\A_K)$ survives $F$-descent, then there exists $\xi \in \HH^1(K,\ker F)$ such that for each $v$, $\mu_{K_v}(x_v) = \Phi_K(\xi) \in \Omega^{\oplus g}_{K/\F} \subset \Omega_{K_v/\F}^{\oplus g}$. In this case the point $(f_1(x_v) : \dots : f_g(x_v))$ has coordinates in $K$ and is independent of $v$. This immediately implies that $(x_v) \in C(K)$ 
%
%
unless $C$ is hyperelliptic and $C \to \PP^{g-1}$ is not ramified above $P$. In this case the fiber of $C \to \PP^{g-1}$ above $P$ is a locally trivial torsor under $\Z/2\Z$ and, hence, consists of a pair of global points $Q,Q' \in C(K)$ interchanged by the hyperelliptic involution. Since $Q + Q'$ is linearly equivalent to an $\F$-rational divisor it lies in $F(J^{(-1)}(\F)) \subset \ker(\mu_K)$ and so $\mu_K(Q) = -\mu_K(Q')$. Since $p$ is odd, this shows that $\mu_K(Q) \ne \mu_K(Q')$. Since $x_v \in \{ Q,Q'\}$ and $\mu_{K_v}(x_v)$ does not depend on $v$, we conclude that all $x_v$ must be equal and so $(x_v) \in C(K)$. 

For hyperelliptic curves in characteristic $2$ a similar argument using $F^2$-descent taking values in a module over the ring $W_2(K)$ of length two Witt vectors gives the result (we omit details).
\end{proof}

\begin{Remark}
Part of the argument of Lemma \ref{lem:Frob} comes from the proof of \cite[Theorem 4.2.1]{ARM} where it is used 
in a (slightly) different context. 
\end{Remark}

\begin{Remark}
In the case $p = 2$ and $C$ is hyperelliptic the proof shows that an adelic point $(x_v)$ surviving $(C',\ker(F))$ which does not lift to $C'$ has $x_v \in C(K)$ for each $v$ and in particular that $C(K) \ne \emptyset$. The issue is that the various $x_v$ may differ from one another by the hyperelliptic involution and so $(x_v)$ may not be global.
\end{Remark}

\begin{proof}[Proof of Theorem~\ref{Thm:FrobDescent}]
	Let $(x_v) \in C(\A_K)^{F^\infty}$. In particular, $(x_v)$ survives $(C',\ker(F))$. If $\mu_{K_v}(x_v) \ne 0$, then $(x_v)$ is global by the lemma. Otherwise $(x_v) \in F(J^{(-1)}(\A_{K}))$ in which case $(x_v)$ lifts to some $(y_v) \in C^{(-1)}(\A_{K})$. Since $\ker(F)$ has nontrivial rational points defined over the separable closure of $K$, \cite[Main Theorem]{GA-T} gives that $\Sha^1(K,\ker(F)) = \ker\left(\HH^1(K,\ker(F)) \to \prod_v \HH^1(K_v,\ker(F))\right) = 0$. Hence $(x_v)$ does not lift to any nontrivial twist of $(C',\ker(F))$. It follows that $(y_v) \in C^{(-1)}(\A_K)^{F^\infty}$. This argument may be iterated, so we conclude that $(x_v)$ is either global or arbitrarily divisible by Frobenius, hence in $C(\A_{K,\F})$ by \cite[Lemma 2.13]{CVV}. This proves that $C(\A_K)^{F^\infty} \subset C(K) \cup C(\A_{K,\F})$. For the reverse inclusion note that $C'(\A_{K,\F}) = C^{(-1)}(\A_{K,\F})$ and the map $F : C'(\A_{K,\F}) \to C(\A_{K,\F})$ agrees with relative Frobenius map $C^{(-1)}(\A_{K,\F}) \to C(\A_{K,\F})$ which is surjective. Hence $C(\A_{K,\F}) \subset C(\A_K)^{F^\infty}$.
\end{proof}

\begin{Remark}
It follows from the proof of Theorem~\ref{Thm:FrobDescent} that the subset of $C(K)$ corresponding to
non-constant maps $D \to C$ is characterized as the set of adelic points surviving ${F^\infty}$-descent that, 
at some stage, lift to a non-trivial torsor.
\end{Remark}

\begin{Corollary}
All nontrivial twists of the Frobenius torsor $(C',\ker(F)) \in \Cov(C_K)$ satisfy the Hasse principle, i.e., if $Y \to C_K$ is a twist whose class in $\HH^1(K,\ker(F))$ is nontrivial and $Y(\A_K) \ne \emptyset$, then $Y(K) \ne \emptyset$.
\end{Corollary}

\begin{proof}
	If $Y$ is such a torsor and contains an adelic point, then this point maps to an adelic point $(x_v) \in C(\A_K)$ unobstructed by $F$. As in the previous proof $\HH^1(K,\ker(F)) \to \prod_v\HH^1(K_v,\ker(F))$ is injective by \cite[Main Theorem]{GA-T}. Since $Y$ is a nontrivial twist, this implies that there is some $v$ such that $\mu_{K_v}(x_v) \ne 0$. Hence, $C$ contains a global point that is the image of a $K_v$-point on $Y$ by the lemma (and the remark above). Since $Y \to C$ is purely inseparable, this implies that $Y$ contains a global point as well.
\end{proof}

Note that whether or not the trivial torsor satisfies the Hasse principle
depends on whether or not $C$ itself does.
	
	\section{Etale abelian descent obstruction}\label{sec:etale}
	
	In this section we show that, at the level of reduced adelic points, all of the information given by the Brauer group can be obtained from finite abelian and \'etale torsors. We use $X(\A_{K,\F})^\star$ to denote $X(\A_K)^\star \cap X(\A_{K,\F})$.
	
	\begin{Proposition}\label{prop:etale}
		Suppose $X$ is a closed subvariety of an abelian variety. Then $X(\A_{K,\F})^{\isog} = X(\A_{K,\F})^{\etaleisog}$.
	\end{Proposition}
	
	\begin{proof}
		Suppose $P \in X(\A_{K,\F})^{\etaleisog}$ and let $(X',G)$ be the pullback of some isogeny $\phi : A' \to A$. We sill show that $P$ survives $(X',G)$. Since $G/\F$ is abelian it decomposes as a direct product $G = G_c \times G_e$ of a connected group scheme and an \'etale group scheme \cite[Proposition 11.3]{MilneAG}. This gives rise to an etale torsor $(X'/G_c,G_e)$ and an inseparable torsor $(X'/G_e, G_c)$, and it suffices to show that $P$ survives both. It survives the first by assumption. Evaluation of the second torsor at $P$ gives an element of $\prod_{v \in D^1}\HH^1(\F_v,G_c)$. We claim that $\HH^1(\F_v,G_c) = 0$, so $P$ lifts to $(X'/G_e, G_c)$ as well.
		
		To establish the claim, we use that there are abelian varieties $A,B$ over $\F$ fitting into an exact sequence
		\begin{equation}\label{eq:ES}
			0 \to G_c \to A \to B \to 0
		\end{equation}
		(See \cite[Appendix A]{MilneADT}). In particular, $A$ and $B$ are isogenous so $A(\F_v)$ and $B(\F_v)$ have the same cardinality by a celebrated result of Tate \cite{TateEndomorphisms}. Then $A(\F_v) \to B(\F_v)$ is surjective, being a homomorphism of finite groups of the same size with kernel $G_c(\F_v) = 0$. On the other hand $\HH^1(\F_v,A) = 0$ by Lang's theorem, so the long exact sequence of cohomology groups assoicated to \eqref{eq:ES} gives $\HH^1(\F_v,G_c) = 0$.
	\end{proof}
	
	\begin{Theorem}\label{thm:BrAKF}
		If $X/\F$ is either a curve or an abelian variety, then $X(\A_{K,\F})^{\Br} = X(\A_{K,\F})^{\fab} = X(\A_{K,\F})^{\etaleisog}$. If $X$ is a curve, then also $X(\A_{K,\F})^{\Br} = X^\textup{MW-Sieve}$. 
	\end{Theorem}
	
	\begin{proof} 
		The first statement follows from Propositions~\ref{prop:BMandDescent} and~\ref{prop:etale}. The proof of \cite[Proposition 2.12]{CVV} gives the second.		
	\end{proof}
	
	This proves~Theorem~\ref{Thm:EtDescent} and, consequently, Theorem~\ref{Thm1}, which has the following Corollary.
	
	\begin{Corollary}\label{cor:conjectures}
		Suppose $C/\F$ is a curve. Then $\overline{C(K)} = C(\A_K)^{\Br}$ in $C(\A_K)$ if and only if $C^\textup{MW-Sieve} = \overline{r(C(K))}$ in $C(\A_{K,\F})$.
	\end{Corollary}
	
	\begin{proof}
		Assume $\overline{C(K)} = C(\A_K)^{\Br}$ and let $P \in C^\textup{MW-Sieve}$. By Theorem~\ref{Thm1} there exist a sequence of elements $P_n \in C(K)$ converging to $P$ in $C(\A_K)$. Since $r : C(\A_K) \to C(\A_{K,\F})$ is continuous, $r(P_n) \to r(P) = P$, so $P \in \overline{r(C(K))}$. 
		
		Conversely, suppose $C^\textup{MW-Sieve} = \overline{r(C(K))}$ and let $P \in C(\A_K)^{\Br}$. By Theorem~\ref{Thm1} we have $P \in C(K)$ or $P \in C^\textup{MW-Sieve}$. In the former case $P$ obviously lies in the closure of $C(K)$, so suppose $P \in C^\textup{MW-Sieve}$. By assumption there are $P_n \in C(K)$ such that $r(P_n) \to P$ in the subspace $C(\A_{K,\F}) \subset C(\A_K)$. By \cite[Lemma 2.13]{CVV} the sequence $F^{n!} : C(\A_K) \to C(\A_K)$ of $n!$-th iterates of the $\F$-Frobenius converges uniformly to $r$. Hence $\lim_n F^{n!}P_n = P$, showing that $P$ lies in the closure of $C(K)$ in $C(\A_K)$.		
	\end{proof}
	
	Theorem~\ref{thm:BrAKF} does not hold if one replaces $X(\A_{K,\F})$ with $X(\A_K)$.
	
	\begin{Proposition}\label{prop:example1}
		Let $X/\F$ be an abelian variety of $p$-rank $0$. There are infinitely many global function fields $K = \F(D)$ such that $X(K) = X(\A_K)^{\Br} \ne X(\A_K)^{\etaleisog}$. 
	\end{Proposition}
	
	\begin{proof}
		Let $(x_v) \in X(\A_K)$ be such that $r(x_v) = 0_X$. For each $v$, the kernel of reduction in $X(K_v)$ is divisible by all integers prime to $p$. So, for any $n$ prime to $p$ we have $(x_v) \in nX(\A_K)$. The assumption on the $p$-rank implies that all (geometrically connected) torsors under \'etale abelian group schemes are dominated by a twist of $[n] : X \to X$ for some $n$ prime to $p$ and so $(x_v) \in X(\A_K)^{\etaleisog}$. In particular, for any $K$ the set $X(\A_K)^{\etaleisog}$ is infinite. However, if $K = \F(D)$ is such that $X$ and $\Jac(D)$ have no isogeny factors in common. Then $X(K) = X(\F)$ is finite.
	\end{proof}

	For a constant variety $X/\F$ let $X(\A_K)^\textup{\'et}$ denote the set of adelic points which survive all torsors under finite \'etale group schemes and set $X(\A_{K,\F})^\textup{\'et} = X(\A_K)^\textup{\'et} \cap X(\A_{K,\F})$. In the case of curves the \'etale torsors do not always cut out the set of rational points, even if one allows torsors under finite non-abelian group schemes that are not required to be geometrically connected. 
	
	\begin{Proposition}\label{prop:example2}
		There exists a curve $C/\F$ of genus $2$ such that $C(K) = C(\A_K)^{\Br} \ne C(\A_K)^{\etale}$.
	\end{Proposition}	
	
	\begin{proof}
		Take $K=\F_p(t)$, $p \ne 2,5$ and $X: y^2=x^5+1$, so $C(K) = C(\F)$ and $\overline{J(K)} = J(K) = J(\F)$ in $J(\A_K)$. By \cite[Prop. 4.6]{PoonenVoloch}, which applies since $\Sha(K,J)$ is finite by \cite[Theorem 3]{Milne}, we have that $C(\A_K)^{\Br} = \overline{J(K)} \cap C(\A_K) = J(\F) \cap C(\A_K) = C(\F)$ in $J(\A_K)$. Now, for every place $v$ of $K$ choose a local parameter $t_v$ and let $x_v = (t_v,\sum_j \binom{1/2}{j} t_v^{5j})$. Then the adelic point $(x_v)$ survives all finite \'etale torsors, lifting to the same twist that $r((x_v))=(0,1) \in C(\F) \subset C(K)$ does. Indeed, this is a special case of the following proposition.
	\end{proof}

	\begin{Proposition}
		Let $(x_v) \in X(\A_K)$ be an adelic point on a constant variety $X/\F$. Then $(x_v) \in X(\A_K)^\textup{\'et}$ if and only if $r(x_v) \in X(\A_{K,\F})^\textup{\'et}$.
	\end{Proposition}
	
	\begin{proof}
		First suppose $(x_v) \in X(\A_K)^\textup{\'et}$. Let $F \colon X \to X$ be the Frobenius morphism. By functoriality of descent (e.g., \cite[Lemma 5.3(2)]{Stoll}) if $(x_v) \in X(\A_K)^\textup{\'et}$ then $F(x_v)$ is as well. It follows that the same is true for $F^{n!}(x_v)$ for any $n \ge 1$. This sequence converges to $r(x_v)$ in the adelic topology. Since $X(\A_K)^\textup{\'et}$ is closed by \cite[Theorem 8.4.6]{PoonenRatPoints} we have that $(x_v) \in X(\A_{K,\F})^\textup{\'et} = X(\A_K)^\textup{\'et} \cap X(\A_{K,\F})$.
		
		For the converse, suppose $r(x_v) \in X(\A_{K,\F})^\textup{\'et}$ and let $(X',G)$ be an $X$-torsor under the finite \'etale group scheme $G$. Replacing $(X',G)$ by a twist if necessary we may assume $r(x_v)$ lifts to an adelic point on $X'$. The fiber of $X' \to X$ over $x_v \in X(K_v) = X(\calO_v)$ gives a class $T_v$ in $\HH^1(\calO_v,G)$. By Hensel's Lemma (which applies since $G$ is \'etale) the reduction map $\calO_v \to \F_v$ induces an isomorphism $\HH^1(\calO_v,G) \simeq \HH^1(\F_v,G)$. Since $r_v(x_v) \in X(\F_v)$ lifts to $X'_v$, the image of $T_v$ in $\HH^1(\F_v,G)$ is trivial. Hence $T_v$ is trivial and so the fiber of $X'$ above $x_v$ contains a $K_v$-point. Thus $X'$ contains a lift of $(x_v)$.
	\end{proof}

\section{Reduced adelic points and maps}\label{sec:maps}
 
Let $J_D$ be the Jacobian of $D$ and fix an embedding $D \to J_D$ corresponding to a $0$-cycle of degree $1$ (which exists by the Hasse-Weil bounds). Recall that by Lemma~\ref{lem:maps} there is a bijection between the set $X(\A_{K,\F})$ of reduced adelic points and the set $\Map_{G_\F}(D(\Fbar),X(\Fbar))$ of Galois equivariant maps on geometric points. 
 
 \begin{Theorem}\label{thm:abdescA}
 	Let $A/\F$ be an abelian variety. For any pair $(\psi,P) \in \Hom_{G_\F}(J_D(\Fbar),A(\Fbar)) \times A(\F)$ the map $\psi_{|_D}+ P : D(\Fbar) \to A(\Fbar)$ corresponds to a reduced adelic point which survives all torsors in $\Cov(A_K)$. This induces bijections
 	\[
 		\overline{r(A(K))} = A(\A_{K,\F})^{\Br} \leftrightarrow \Hom_{G_\F}(J_D(\Fbar),A(\Fbar)) \times A(\F)\,.
 	\]
 \end{Theorem}
 
 \begin{proof}
 	Let $(x_v) \in A(\A_{K,\F})^{\Br}$ and let $\psi : D(\Fbar) \to A(\Fbar)$ be the corresponding Galois-equivariant map. Extending by linearity we obtain a Galois-equivariant homomorphism $\psi' : Z^0(\Dbar) \to Z^0(\Abar)$ on the groups of $0$-cycles. The map sending a $0$-cycle $z \in Z^0(\Abar)$ to the pair $(\text{sum}(z),\deg(z)) \in A(\Fbar) \times \Z$ induces an isomorphism of the group of $0$-cycles on $\Abar$ modulo Albanese equivalence onto $A(\Fbar) \times \Z$. 
 	
 	We claim that if $z \in Z^0(\Dbar)$ is a principal divisor, then the $0$-cycle $\psi'(z)$ is albanese equivalent to $0$. To see this, suppose $z = \sum n_PP$ where $P$ are geometric points of $D$. Let $F^n$ be a power of the $\F$-Frobenius which fixes all $P$ appearing in the support of $z$. By Theorem~\ref{thm:BrAKF}, $(x_v) \in A(\A_{K,\F})^{\etaleisog}$ so it must lift to a twist of $(1-F^n) : A \to A$. Since $\Sha(K,A)$ is finite \cite[Theorem 3]{Milne} we may choose $n$ large enough so that the twist to which it lifts contains a rational point, and hence be of the form $(1-F^n) + \gamma : A \to A$ for some $\gamma \in A(K)$. Evaluating at a point $P \in D(\Fbar)$ in the support of $z$ we conclude that $\psi(P) = (1-F^n)(P) + \gamma(P) = \gamma(P)$. By linearity it follows that $\psi(z) = \gamma(z)$. By the universal property of the Jacobian of $D$, $\gamma \in A(K) = \Mor_\F(D, A)$ factors through a morphism $\gamma_0 : J_D \to A$ of abelian varieties. Hence, if the image of $z$ in $J_D$ is trivial, its image in $A$ must be trivial as well.
 	
 	The claim established in the previous paragraph implies that $\psi'$ factors through a Galois-equivariant map $\Pic(\Dbar) \to A(\Fbar)\times \Z$. We obtain a pair $(\psi,P) \in \Hom_{G_\F}(J_D(\Fbar),A(\Fbar)) \times A(\F)$ by taking $\psi$ as the restriction of the above map to $\Pic^0(\Dbar) = J_D(\Fbar)$ and defining $P$ by $\psi'(z_1) = (P,1) \in A(\Fbar)\times \Z$, where $z_1 \in Z^0(D)$ is the $0$-cycle of degree $1$ used to embed $D$ in $J_D$. This yields the map $A(\A_{K,\F})^{\Br} \to \Hom_{G_\F}(J_D(\Fbar),A(\Fbar)) \times A(\F)$.
 	
 	Now suppose $(\psi,P) \in \Hom_{G_\F}(J_D(\Fbar),A(\Fbar)) \times A(\F)$. The map $\psi_{|_D}+P : D(\Fbar) \to A(\Fbar)$ corresponds to a reduced adelic point of $A$ and it is easy to check that this yields an inverse to the map constructed above. We will show that the adelic point corresponding to $\psi_{|_D}+P$ lies in the closure of the image of $A(K)$ in $A(\A_{K,\F})$. This suffices to prove the theorem since $A(\A_{K,\F})^{\Br}$ is a closed set containing the image of $A(K)$. Furthermore, we may assume $P = 0$.
 	
 	The homomorphism $\psi : J_D(\Fbar) \to A(\Fbar)$ induces also a morphism $\psi' \in \Hom_{G_\F}(T_\textup{\'et}J_D, T_\textup{\'et}A)$ between the full \'etale Tate modules of $J_D$ and $A$, i.e., $T_\textup{\'et}A = \varprojlim_{n} A(\Fbar)[n]$. Since $\F$ is perfect, the abelian group schemes $A[n]$ split as a direct product of an \'etale and a connected group scheme \cite[Proposition 11.3]{MilneAG}. It follows that the full Tate module (profinite group scheme) splits as $TA = T_\textup{\'et}A \times A_0$, with $A_0$ a connected pro-$p$ group scheme and similiarly for $J_D$. As there are no nontrivial morphisms between $p$-primary \'etale and connected group schemes we obtain a surjective map $\Hom_{\F}(TJ_D,TA) \to \Hom_{G_\F}(T_\textup{\'et}J_D, T_\textup{\'et}A)$. Tate's theorem \cite{TateEndomorphisms, WM} gives an isomorphism $\Hom_\F(J_D,A) \otimes \hat{\Z} \to \Hom_{G_\F}(TJ_D,TA)$. From this it follows that, for every $n$, there is some $\phi_n \in \Hom_\F(J_D,A) \subset A(K)$ which agrees with $\psi$ on $J_D[n](\Fbar)$. Then the sequence $r(\phi_{n!})$ converges in $A(\A_{K,\F})$ to the adelic point corresponding to $\psi$.
 
\end{proof}

  	\begin{Corollary}\label{cor:infiniteimage}
        \label{finite}
		Suppose $\psi \colon D(\Fbar) \to A(\Fbar)$ is the map corresponding to a point $(x_v) \in A(\A_{K,\F})^{\Br}$. Then $\psi$ is either constant, in which case $(x_v) \in A(\F)$, or $\psi$ has infinite image.
	\end{Corollary}

	\begin{proof}
		It is enough to show that the image of the induced map $\psi:J_D(\Fbar) \to A(\Fbar)$ has infinite image. This follows since $J_D(\Fbar)$ is a divisible group and, hence, has no nontrivial finite homomorphic image.
	\end{proof}

 \begin{Corollary}\label{cor:abdescC}
 	Suppose $C/\F$ is a smooth, geometrically irreducible, proper curve with a fixed embedding $C \to J$. There is a bijection
 	\[
 		C(\A_{K,\F})^{\Br} \leftrightarrow \left\{ (\psi,P) \in \Hom_{G_\F}(J_D(\Fbar),J(\Fbar))\times J(\F) \,:\, \psi(D(\Fbar)) \subset C(\Fbar)-P \right \}\,.
 	\]
 \end{Corollary}
 
  \begin{proof}
  	We have $C(\A_{K,\F})^{\Br} = C(\A_{K,\F}) \cap J(\A_{K,\F})^{\Br}$ by Theorem~\ref{thm:BrAKF} and the definition of $C(\A_{K,\F})^{\etaleisog}$. A pair $(\psi,P)$ as in Theorem~\ref{thm:abdescA} corresponds to a reduced adelic point on $C$ if and only if $\psi(D(\Fbar)) + P  \subset C(\Fbar)$. 
 \end{proof}
 
 \begin{Remark}
 	Given a reduced adelic point $(x_v) \in C(\A_{K,\F})^{\Br}$ one can always choose an embedding of $C$ into $J$ such that the corresponding pair $(\psi,P)$ has $P = 0$.
 \end{Remark}
  
 \begin{Remark}
Zilber \cites{Z1,Z2} , resolving a conjecture of Bogomolov, Korotiaev and Tschinkel \cite{BKT}, has shown that if $\psi : J_D(\Fbar) \to J(\Fbar)$ is an isomorphism such that $\psi(D(\Fbar)) = \psi(C(\Fbar))$, then $\psi$ is a morphism of curves composed with a limit of Frobenius maps. In particular, the corresponding reduced adelic point lies in the closure of the image of $C(K)$ in $C(\A_{K,\F})$.
 \end{Remark}

		\begin{Corollary}\label{cor:isogenyfactor}
		Suppose $(x_v) \in C(\A_{K,\F})^{\Br}$ and $J$ has only finitely many abelian subvarieties. Then the induced $\psi :J_D(\Fbar) \to J(\Fbar)$ is either constant or surjective, in which case $J$ is an isogeny factor of $J_D$.
	\end{Corollary}

	\begin{proof}
		By functoriality of the Brauer pairing we have $(x_v) \in J(\A_{K,\F})^{\Br}$. Now $J(\A_{K,\F})^{\Br} = \overline{J(K)} \cap J(A_{K,\F})$ in $J(\A_{K})$ by Proposition~\ref{prop:BMandDescent}, so $(x_v) = \lim \phi_n : D \to J$ for some sequence of $\phi_n \in J(K)$. It suffices to show that the induced maps $\phi_n:J_D \to J$ are eventually surjective. If $\phi_n:J_D \to J$ is not surjective, then $\phi_n(D)$ is contained in a translate of a proper abelian subvariety. The intersection $C \cap (x + A)$ of $C \subset J$ with a translate of a proper abelian subvariety $A$ is finite. If there are only finitely many proper abelian subvarieties, then these intersection numbers are bounded. But then so is $\phi_n(D) \cap C$. This implies $\psi(D(\Fbar))$ is finite, so $\psi$ is constant by Corollary~\ref{cor:infiniteimage}.
	\end{proof}

      \begin{proof}[Proof of Theorem~\ref{thm:Brgenus}]
     		Note that the equality $C(K) = C(\F)$ is trivial and that it then follows from Theorem~\ref{Thm:FrobDescent} that $C(\A_K)^{\Br} = C(\A_{K,\F})^{\Br}$. Hence, it is enough to show that $C(\A_{K,\F})^{\Br} =C(\F)$. Suppose $(x_v) \in C(\A_{K,\F})^{\Br}$ and let $\psi : J_D(\Fbar) \to J(\Fbar)$ be the corresponding homomorphism as in Corollary~\ref{cor:abdescC}. We may choose embeddings $D \subset J_D$ and $C \subset J$, such that $\psi$ restricted to $D(\Fbar)$ is the map corresponding to $(x_v)$. Let $I = \psi(J_D(\Fbar))$. Since $J_D(\Fbar)$ is generated by the divisors of degree $g = g_D$ on $D$, we have $I  \subset W^g(C)$, where $W^g(C)$ is the image of the $g$-th symmetric power of $C$ in $J$ under the map induced by the embedding $C \to J$. As $I$ is a topological subgroup of $J(\Fbar)$ with the Zariski topology, its Zariski closure $\overline{I}$ is an algebraic subgroup of $J$ contained in $W^g(C)$. Since $g = \dim W^g(C) < \dim J$, we have that $C \cap \overline{I}$ is a proper algebraic subset of $C$ and, hence, is finite. As this intersection contains $\psi(D(\Fbar))$ we conclude by applying Corollary~\ref{finite}.
      \end{proof}	

\section*{Acknowledgements}
The authors were supported by the Marsden Fund Council administered by the Royal Society of New Zealand. 
Part of this research was carried out during the trimester ``Reinventing Rational Points'' at the Institut Henri Poincar\'e (IHP) and the authors would like to thank the IHP and the organizers of the trimester.
They would also like to thank Damian R\"ossler for discussions related to the topic of this paper.



\end{document}